\documentclass{amsart}

\newcommand{\cL}{\mathcal{L}}
\newcommand{\cM}{\mathcal{M}}
\newcommand{\cO}{\mathcal{O}}

\newcommand{\cS}{\mathcal{S}}

\newcommand{\cW}{\mathcal{W}}

\newcommand{\m}{\to} \newcommand{\Q}{\mathbb{Q}}\newcommand{\R}{\mathbb{R}}\newcommand{\C}{\mathbb{C}}
\newcommand{\p}{Poincar\'e }

\newcommand{\bC}{\mathbb{C}}

\newcommand{\bN}{\mathbb{N}}

\newcommand{\bQ}{\mathbb{Q}}\newcommand{\bR}{\mathbb{R}}

\newcommand{\bZ}{\mathbb{Z}}

\usepackage[dvipsnames,svgnames,x11names,hyperref]{xcolor}
\usepackage{url,graphicx,verbatim,amssymb,enumerate,stmaryrd}
\usepackage[pagebackref,colorlinks,citecolor=Blue,linkcolor=Blue,urlcolor=Blue,filecolor=Blue]{hyperref}
\usepackage{microtype,multicol}
\usepackage[all]{xy}
\usepackage[capitalise]{cleveref}

\newtheorem{theorem}{Theorem}[section]
\newtheorem{lemma}[theorem]{Lemma}
\newtheorem{proposition}[theorem]{Proposition}

\newtheorem{conjecture}[theorem]{Conjecture}
\theoremstyle{definition}
\newtheorem{definition}[theorem]{Definition}
\newtheorem{example}[theorem]{Example}
\newtheorem{remark}[theorem]{Remark}

\newcommand{\mr}[1]{{\rm #1}}
\newcommand{\fS}{\mathfrak{S}}

\newcommand{\syms}{S}

\newcommand{\symw}{W}
\newcommand{\stab}{t}
\newcommand{\OO}{\mathcal O}

\title{Homological stability for symmetric complements}
\author{Alexander Kupers}
\thanks{Alexander Kupers is supported by a William R. Hewlett Stanford Graduate Fellowship, Department of Mathematics, Stanford University, and was partially supported by NSF grant DMS-1105058.}
\author{Jeremy Miller}
\author{TriThang Tran}
\date{\today}

\begin{document}

\begin{abstract}Conjecture F from \cite{VW} states that the complements of closures of certain strata of the symmetric power of a smooth irreducible complex variety exhibit rational homological stability. We prove a generalization of this conjecture to the case of connected manifolds of dimension at least 2 and give an explicit homological stability range.\end{abstract}

\maketitle

\section{Introduction}

The goal of this paper is to prove a generalization of a conjecture of Vakil and Wood (Conjecture F of \cite{VW}). This conjecture concerns homological stability for certain subspaces of symmetric powers, defined as the complements of closures of certain strata. Here a sequence of spaces $X_k$ is said to have homological stability if the homology groups $H_i(X_k)$ are independent of $k$ for $k \gg i$.

We begin by defining the relevant subspaces of symmetric powers. Let $\mr{Sym}_k(M)$ denote the symmetric power $M^k/\fS_k$ of a space $M$. Here $\fS_k$ denotes the symmetric group on $k$ letters acting by permuting the terms. We think of this as the space of \emph{configurations} of \emph{particles} in $M$ with multiplicities summing to $k$. These are allowed to collide and then the multiplicities add up. To any such configuration in $\mr{Sym}_k(M)$ we can associate a way of writing the number $k$ as a sum of positive integers by recording the multiplicity of each particle. Such a sum is called a \textit{partition} of $k$. For example, to the element $\{m_1,m_2,m_2,m_3\} \in \mr{Sym}_4(M)$ with $m_1,m_2,m_3$ distinct we can associate the partition $4=1+1+2$. Using this one can define the following subspaces of $\mr{Sym}_k(M)$.

\begin{definition}\label{SDW} Let $\lambda$ be a partition of $k$. \begin{enumerate}[(i)] \item Let $S_\lambda(M)$ be the subspace of $\mr{Sym}_k(M)$ consisting of elements that have associated partition equal to $\lambda$. We call this the \textit{stratum corresponding to $\lambda$}.
\item Let $W_\lambda(M)$ be the complement of the closure of $S_\lambda(M)$ in $\mr{Sym}_k(M)$. We call this the \textit{symmetric complement associated to $\lambda$}.
\end{enumerate}
\end{definition}

An element of $S_{\lambda}(M)$ can be viewed as a configuration of distinct particles in $M$ labeled by $\lambda$. For  $\lambda=4+4+5$ for example, $S_{\lambda}(M)$ is the configuration space of $3$ distinct particles in $M$, two of which are labeled by the number $4$ and the remaining particle labeled by the number $5$. The two particles labeled by the number $4$ are indistinguishable from each other, but can be distinguished from the particle labeled by the number $5$. If $\lambda = 1+ \cdots +1$ is a partition of $k$, then $S_{\lambda}(M)$ is the space of configurations of $k$ distinct unordered particles in $M$, often denoted $C_k(M)$. In general $S_\lambda(M)$ is homeomorphic to one of the colored configuration spaces considered in \cite{Ch}.

One can think of $W_\lambda(M)$ as those elements of $\mr{Sym}_k(M)$ that cannot be  made to have associated partition $\lambda$ by an arbitrarily small perturbation. For example, an element of $\mr{Sym}_7(M)$ is in $W_\lambda(M)$ with $\lambda=1+1+1+2+2$ if all but possibly one of particles have multiplicity equal to one and no particles have multiplicity four or higher. Partitions of $k$ of the form $1+\cdots +1 +(c+1)$ yield the spaces homeomorphic to the bounded symmetric powers $\mr{Sym}^{\leq c}_k(M)$, defined as the subspace of $\mr{Sym}_k(M)$ where no particle has multiplicity greater than $c$. For $c=1$, this is simply $C_k(M)$ the configuration space of $k$ distinct unordered particles in $M$. 

Conjecture F pertains to the limiting behavior of symmetric complements as the number of particles increases. To state it we need the following construction. If $\lambda$ is a partition of $k$ we can obtain from it a partition of $k+1$ by adding another $1$ to $\lambda$. More generally we can add $j$ additional $1$'s to obtain a partition of $k+j$, which we denote by $1^j\,\lambda$. In other words, if $\lambda=m_1+\cdots +m_i$, then $1^j \lambda$ is the partition $1+ \cdots +1+m_1+\cdots + m_i$ where there are $j$ additional $1$'s. In \cite{VW}, Vakil and Wood made the following conjecture. 

\begin{conjecture}[Conjecture F]
For any irreducible smooth complex variety $X$, $\dim H_i(W_{1^j \lambda}(X);\bQ) = \dim H_i(W_{1^{j+1} \lambda}(X);\bQ)$ for $j \gg i$.
\end{conjecture}

We prove this conjecture, generalize it to the case of arbitrary connected smooth manifolds of dimension at least 2 (henceforth all manifolds are assumed smooth) and give an explicit homological stability range. That is, we prove the following theorem.

\begin{theorem} \label{main} Let $M$ be a connected manifold of dimension $d\geq 2$ and $\lambda$ a partition of $k$. We have that
\[H_i(W_{1^j \lambda}(M);\bQ) \cong H_i(W_{1^{j+1} \lambda}(M);\bQ)\]
for $i \leq j+k-1$. 
\end{theorem}

We actually give a better range that depends on $M$ as well as $k$, described by functions $f_{M,k}:\bN_0 \m \bN_0$ defined in Equation \ref{eqnfmlambda} on page \pageref{eqnfmlambda}, where $\bN_0$ denotes the non-negative integers. The use of rational coefficients is essential in many parts of the argument but not all. See Remark \ref{remintegral} for a discussion of what results hold with integral coefficients.

In general, the isomorphism of Theorem \ref{main} is given by a transfer map which is described in Definition \ref{deftransfer}. When $M$ is the interior of a manifold with non-empty boundary one can also define a stabilization map $t: W_{1^j \lambda}(M) \m W_{1^{j+1} \lambda}(M)$ given by ``bringing a particle in from infinity,'' described in Definition \ref{defstab}. The stabilization map induces an isomorphism on rational homology in the same range as the transfer map. This is not a coincidence, but part of the proof. Our result uses homological stability for configuration spaces of unordered distinct particles as input \cite{Ch, RW, K}. It does not use homological stability for bounded symmetric powers and hence gives a new proof of Theorem 1.6 of \cite{kupersmillercompletions} with an improved range. Indeed, the range one obtains for $\mr{Sym}_k^{\leq c}(M)$ is $* \leq k-1$ if $M$ is a surface and $* \leq k$ otherwise.

\subsection{Motivic motivation for Conjecture F}

Conjecture F was inspired by Theorem 1.30 of \cite{VW} which can be thought of as Conjecture F's motivic analogue. As an abelian group, the Grothendieck ring of varieties $\cM$  is defined as the quotient of the free abelian group on the set of isomorphism classes of finite type schemes over a field $\mathbb K$, modulo the relation $[X] = [Z] + [X\backslash Z]$ whenever $Z \subset X$ is a closed subscheme. The ring structure is induced by Cartesian product. Let  $\mathbf{L}$ denote the affine line and let $\cM_\mathbf{L}$ denote the localization of $\cM$ obtained by inverting $[\mathbf{L}]$. The ring $\cM$ has a filtration induced by the dimension of the scheme. This dimensional filtration extends to a filtration of $\cM_\mathbf{L}$ and we denote the completion with respect to this filtration by $\widehat{\cM_\mathbf{L}}$. Let $X$ be an irreducible stably rational variety of dimension $d$. A consequence of Theorem 1.30 and Motivation 1.26 (i) of \cite{VW} is that the sequence $[W_{1^j \lambda}(X)]/[\mathbf{L}]^{dj}$ converges in $\widehat{\cM_\mathbf{L}}$ with respect to the topology induced by the dimensional filtration. 

From now on, we restrict to the case $\mathbb K =\bC$. Let $K_{HS}$ denote the Grothendieck group of the category of mixed Hodge structures. Define $HS:\cM \m K_{HS}$ by the formula $[X] \m \sum_i (-1)^i [H^i_c(X;\Q)]$. This extends to a continuous ring homomorphism $HS:\widehat{\cM_\mathbf{L}} \m \widehat{ K_{HS}}$ with $\widehat{ K_{HS}}$ a completion of $ K_{HS}$. Theorem 1.30 and Motivation 1.26 (ii) of \cite{VW} imply that $HS([W_{1^j \lambda}(X)]/[\mathbf{L}]^{dj})$ converges in $\widehat{ K_{HS}}$ for all smooth irreducible complex varieties $X$. The ring $\widehat{ K_{HS}}$ is called the completed ring of virtual Hodge structures and the map $HS$ is an example of a motivic measure. Theorem 1.30 of \cite{VW} is more general than these two results about convergence in $\widehat{\cM_\mathbf{L}}$ and $\widehat{ K_{HS}}$ and implies convergence after applying an arbitrary motivic measure under the assumption of motivic stability of symmetric powers. 

Conjecture F is part of a larger question concerning the relationship between homological stability and stability in the Grothendieck ring of varieties. For many sequences of varieties with homological stability, Vakil and Wood were able to prove that the corresponding elements in the Grothendieck ring converge; after applying a motivic measure, they were able to prove this in even more cases. For $W_{1^j \lambda}(X)$, the corresponding elements in the Grothendieck ring often stabilize and always stabilize in the ring of virtual Hodge structures. However, homological stability was previously not known. Conjecture F is obtained from the idea that one should expect homological stability in situations where there is stability in the Grothendieck ring or ring of virtual Hodge structures and vice versa. 

There is in fact a close but not exact relationship between the singular homology of a complex variety and its corresponding element in the Grothendieck ring. For example, for smooth and proper varieties, the element in $ \widehat{ K_{HS}}$ and hence the element in $\widehat{\cM_\mathbf{L}} $ determines the rational Betti numbers.  Using motivic zeta functions and a heuristic they dub ``Occam's razor for Hodge structures,'' Vakil and Wood developed a procedure for predicting rational Betti numbers from elements of the Grothendieck ring. This heuristic was designed to explain the apparent correlation between the two types of stability and give a prediction of the stable homology. Although there are some examples where Vakil and Wood's predictions of the stable homology are incorrect, see e.g. \cite{kupersmillernote, tommasi}, we know of no examples where they make incorrect predictions regarding whether a sequence of spaces has rational homological stability. It would be interesting to know under what conditions this motivic convergence is in fact equivalent to rational homological stability.

What about the stable homology? As noted before, Vakil and Wood's algebro-geometric approach of motivic zeta functions and ``Occam's razor for Hodge structures'' does not always correctly predict the limiting rational Betti numbers. The homotopy theoretic technique of ``scanning'' \cite{Se3, Mc1} also fails to provide an answer in general, though for bounded symmetric powers it can still be made to work \cite{K2, kupersmillercompletions}. The problem is that in general the spaces $W_{1^j \lambda}(M)$ cannot be characterized by local conditions. We would be interested in any new techniques that shed light onto the stable homology groups of the spaces $W_{1^j \lambda}(M)$ as well as their stable mixed Hodge structures (see Remark \ref{hodgeremark}).

\subsection{Outline} In Section \ref{secCS}, we describe a spectral sequence for computing compactly supported cohomology associated to an open filtration. In Section \ref{evenSec}, we prove Theorem \ref{main} when $M$ is an even dimensional orientable manifold which is the interior of a  manifold with non-empty boundary. In odd dimensions or when $M$ is not orientable, there are extra complications stemming from the fact that $W_{1^{j} \lambda}(M)$ is not orientable. In Section \ref{secNON}, we describe how to modify the proof to address these orientation issues. In Section \ref{secpuncturing}, we discuss how to remove the hypothesis that $M$ is the interior of a manifold with non-empty boundary.  

\subsection{History of this paper}The first two authors and the third author independently proved Conjecture F in \cite{kupersmillerconjecturef} and \cite{transymmetric} respectively. This paper was obtained by merging those two preprints.

\subsection{Acknowledgements} We would like to thank Martin Bendersky, Tom Church, S{\o}ren Galatius, Daniel Litt, Sam Nariman, Ravi Vakil, Alexander Voronov, Nathalie Wahl, Craig Westerland, and Melanie Wood for helpful discussions. Additionally, we thank the anonymous referee for many helpful suggestions and corrections.

\section{Compactly supported cohomology}\label{secCS}

The use of compactly supported cohomology for proving homological stability results was pioneered by Arnol'd in \cite{Ar}. In this section we review basic properties of compactly supported cohomology and describe a spectral sequence associated to an open filtration. If $N_j$ is a sequence of orientable manifolds each of dimension $n_j$, then $H_*(N_j) \cong H_*(N_{j+1})$ for $* \leq r_j$ if and only if $H^{*}_c(N_j) \cong H^{*+n_{j+1}-n_j}_c(N_{j+1})$ for $* \geq n_j-r_j$. Thus, for orientable manifolds homological stability in a range bounded from above is equivalent to stability in a range bounded from below for compactly supported cohomology. It is sometimes more convenient to use compactly supported cohomology for the following reason. Suppose we have a sequence $X_j$ of filtered spaces; then one can often leverage stability for the filtration differences $F_p X_j \backslash F_{p-1} X_j$ to prove stability for the $X_j$. This can be done using the following long exact sequence in compactly supported cohomology (for example see III.7.6 of \cite{iversen}) and the subsequent spectral sequence derived from it. 

\begin{remark}\label{remarksheaf} Our reference \cite{iversen} uses sheaf cohomology, but Theorem III.1.1 of \cite{bredonsheaf} implies that for locally compact locally path-connected Hausdorff spaces, we have that compactly supported sheaf cohomology with coefficients in a locally constant local system coincides with compactly supported singular cohomology in that locally constant local system.\end{remark}

\begin{proposition} \label{exactseq} Let $X$ be a locally compact and locally path-connected Hausdorff space and $C \subset X$ a closed subspace that is also locally path-connected. Denote its open complement $U = X \backslash C$. There is a long exact sequence in compactly supported cohomology
\[\cdots \to H^*_c(U) \to H^*_c(X) \to H^*_c(C) \to H^{*+1}_c(U) \to \cdots\]
The same holds for compactly supported cohomology with coefficients in a local coefficient system on $X$.
\end{proposition}

Here we remark that open subsets of a locally compact and locally path-connected Hausdorff space are again locally compact and path-connected Hausdorff spaces. Closed subsets are similarly locally compact and Hausdorff, and our additional hypothesis on $C$ guarantees it is furthermore locally path-connected. Iterating this proposition gives the following spectral sequence associated to a finite open filtration. We are unaware of a reference so we sketch a proof.

\begin{proposition}\label{specOpen} Let $X$ be a locally compact and locally path-connected Hausdorff space and 
\[\ldots = U_{M-1} = U_{M} = \emptyset \subset U_{M+1} \subset \ldots \subset X = U_N =  U_{N+1} = \ldots\]
be an increasing sequence of open subsets of $X$, such that each $U_p \backslash U_{p-1}$ is locally path-connected. Then there is a spectral sequence  converging to $H^{p+q}_c(X)$ with $E^1$-page given by
\[E^1_{p,q} = H^{p+q}_c(U_{p}\backslash U_{p-1})\]
There is a similar spectral sequence for compactly supported cohomology with coefficients in any local system on $X$. It is natural with respect to open embeddings compatible with the filtrations.
\end{proposition}

\begin{proof}The idea is to splice together the long exact sequences for the inclusions of a closed set and its complement $U_i \backslash U_{i-1} \hookrightarrow U_i  \hookleftarrow U_{i-1}$, given by Proposition \ref{exactseq}
\[\cdots \to H^*_c(U_{i-1}) \to H^*_c(U_i) \to H^*_c(U_i \backslash U_{i-1}) \to H^{*+1}_c(U_{i-1}) \to \cdots\]
and consider the following exact couple:

\noindent\begin{minipage}{.5\linewidth}
\[\xymatrix{A \ar[rrrr]^i_{(1,-1)} & & & & A \ar[dll]^j_{(0,0)} \\
   & & E \ar[ull]^k_{(-1,2)} & & }\]
\end{minipage}%
\begin{minipage}{.3\linewidth}
  \begin{align*}A_{p,q} &= H^{p+q}_c(U_p)\\
  E_{p,q} &= H^{p+q}_c(U_p \backslash U_{p-1})\end{align*}
\end{minipage} 

\noindent with $(a,b)$ denoting the shift in bigrading. Here $i$ is the sum of the maps $H^{*}_c(U_{p-1}) \to H^{*}_c(U_p)$ induced by the inclusion of open subsets, $j$ is induced by the restriction map $H^*_c(U_p) \to H^*_c(U_p \backslash U_{p-1})$ along the closed inclusion $ U_p \backslash U_{p+1} \hookrightarrow U_p$ and $k$ is given by the boundary maps $H^*_c(U_p \backslash U_{p-1}) \to H^{*+1}_c(U_{p-1})$. 

The general machinery of exact couples gives us a spectral sequence with 
\[E^1_{p,q} = H^{p+q}_c(U_p \backslash U_{p-1}) \]
which has differentials $d_r$ of bidegree $(-r,r+1)$. To check that this converges and compute what it converges to, we note that there are only finitely many $p$ such that $i: A_{p,q} \to A_{p+1,q-1}$ is not an isomorphism. In particular, $A_{p,q}$ is $0$ for $p$ sufficiently small and $A_{p,q} = H^{p+q}_c(X)$ for $p$ sufficiently large. In this situation we can define $A_{-\infty,*}$ and $A_{\infty,*}$ in terms of $A_{p,*}$ for $p$ sufficiently small and large respectively. It suffices to note that a spectral sequence of an exact couple with such a finite filtration has $E^\infty$-page isomorphic to the associated graded of the filtration of $A_{\infty,*}$ if $A_{-\infty,*} = 0$. In our case $A_{\infty,*} \cong  H^*_c(X)$ and $A_{-\infty,*} = 0$ and we get the desired result.
\end{proof}

\section{The proof for open orientable manifolds of even dimension}\label{evenSec}

In this section we prove homological stability for $\symw_{1^j\lambda}(M)$ where $M$ is a connected orientable manifold of even dimension $d = 2n \geq 2$ that is the interior of a manifold with non-empty boundary.

We start by defining the stabilization map. Let $M$ be the interior of $\bar{M}$, a connected manifold with non-empty boundary $\partial \bar{M}$. We do not require $\bar{M}$ to be compact. Let $D^k$ denote the open $k$-disk and $\bar{D}^k$ denote the closed $k$-disk. Pick an embedding $\phi: \bar{D}^{d-1} \hookrightarrow \partial \bar{M}$ and a homeomorphism \[\psi: \mr{int}(\bar{M} \cup_{\phi} \bar{D}^{d-1} \times [0,1)) \to M\] whose inverse is isotopic to the inclusion of $M$ into  $\mr{int}(\bar{M} \cup_{\phi} \bar{D}^{d-1} \times [0,1))$. 

\begin{definition} \label{defstab}
The stabilization map $t: \bR^{d} \times W_{1^j \lambda}(M) \to W_{1^{j+1} \lambda}(M)$ is defined as follows: For $z \in \bR^{d}$ and $\xi \in W_{1^j \lambda}(M)$, let $\xi \cup z$ be the element of $W_{1^{j+1} \lambda}(\mr{int}(\bar{M} \cup_{\phi} \bar{D}^{d-1} \times [0,1)))$ given by $\xi$ in $M$ and $z$ in $ \bR^{d} \cong D^{d-1} \times (0,1)$. Define $t$ by the formula $t(z,\xi) = \hat \psi(\xi \cup z)$ where \[\hat \psi:W_{1^{j+1} \lambda}(\mr{int}(\bar{M} \cup_{\phi} \bar{D}^{d-1} \times [0,1))) \to W_{1^{j+1} \lambda}(M)\] is the map induced by applying $\psi$ to every particle in the configuration.
\end{definition}

Note that this map depends on a choice of embedding and homeomorphism. However, up to homotopy, it only depends on the choice of $\bar{M}$ and the choice of the component of $\partial \bar{M}$. Note that up to proper homotopy, the stabilization map also depends on a choice of orientation of the component of $\partial \bar{M}$, provided  that component is orientable. Since $\bR^{d}$ is contractible, $H_*(\bR^{d} \times W_{1^j \lambda}(M)) \cong H_*( W_{1^j \lambda}(M)) $.  However, our proof of homological stability will use compactly supported cohomology and from that perspective, the copy of $\bR^{d}$ is relevant. In particular, it makes the stabilization map an open embedding so it induces a map on compactly supported cohomology. In a similar fashion, one can define stabilization maps for the strata $S_{\lambda}(M)$. To state the main result of this section, we need the following definition.

\begin{definition}\label{defcondasta} A manifold $M$ is said to satisfy condition $(*)_a$ for $a<\dim M-1$ if $M$ is orientable and $\tilde H_i(M;\bQ) = 0$ for $ i \leq a$.
\end{definition} 

The goal of this section is to prove the following proposition.

\begin{proposition}\label{propevenopen} Let $M$ be a connected orientable manifold of even dimension $d = 2n \geq 2$ that is the interior of a manifold with non-empty boundary and $\lambda$ be a partition of $k$. The stabilization map
	\[ \stab_* : H_i(\R^d \times \symw_{1^j \lambda}(M); \mathbb{Q}) \rightarrow H_i(\symw_{1^{j+1}\lambda}(M); \mathbb{Q}) \]
is an isomorphism for $i < f_{M, k}(j)$ and a surjection for $i = f_{M, k}(j)$, where $f_{M, k}(j)$ is defined as follows\footnote{The cases in the definition of $f_{M,k}$ are not mutually exclusive. If a manifold is covered by more than one of the cases, take the maximum range.}:
\begin{equation}\label{eqnfmlambda} 
f_{M,k}(j) = 
	\begin{cases} j + k	 	& \text{if $\dim M = d >2$}; \\
	j + k 	-1					& \text{if $\dim M = 2$ and $M$ is orientable}; \\
	j+k &\text{if $\dim M = 2$  and $M$ is not orientable;} \\
	\min(a+1,\frac{\dim M}{2})(j+k)-1 				& \text{if condition $(*)_a$ holds for $a \geq 1$} \\
 \end{cases}
\end{equation}
\end{proposition}

In \cref{lemtransferinverse}, we will show that the stabilization map is always injective on homology which will imply that it is an isomorphism for $i=f_{M,k}(j)$.

Our approach for proving \cref{propevenopen} will be to filter $\symw_{1^j \lambda}(M)$ such that the filtration will consist of disjoint unions of spaces of the form $S_{\lambda'}(M)$. We use this to produce a spectral sequence for compactly supported cohomology using \cref{specOpen}. Using homological stability for the filtration differences and rational Poincar\'{e} duality we will see that the $\symw_{1^j \lambda}(M)$ satisfy stability in compactly supported cohomology. Again by rational Poincar\'{e} duality this will be equivalent to showing that $\symw_{1^j \lambda}(M)$ satisfies rational homological stability.

Using the homological stability results of Church \cite{Ch} and Randal-Williams \cite{RW}, we prove that the stabilization map induces isomorphisms on homology for each of the components of each of the filtration differences. This is sometimes referred to as homological stability for colored configuration spaces. A version of it appears as Theorem 5 of \cite{Ch}.

\begin{lemma}\label{lemOpenstratastab} Let $\lambda$ be a partition with $i$ $1$'s. If $M$ is connected and orientable, the map \[t_*: H_*(\R^{d} \times S_{\lambda}(M);\bQ) \to H_*(S_{1\,\lambda}(M);\bQ)\] is an isomorphism in the following ranges: (i) $* \leq i$ if $M$ is of dimension $d>2$, (ii) $* < i$ if $M$ is of dimension $d=2$, and (iii) $* < (a+1)i$ if condition $(*)_a$ holds.\end{lemma}

\begin{proof}Let $\lambda'$ be the partition such that $1^i \lambda' =\lambda$ and let $r$ be the cardinality of $\lambda'$, i.e. $\lambda'$ is given by $m_1 + \cdots + m_r$ with $m_i \geq 2$ for all $1 \leq i \leq r$. There is a fiber sequence: \[S_{1^{i}}(M \backslash \{r \text{ points}\}) \to S_{\lambda}(M) \to S_{\lambda'}(M)\] where the map $S_{ \lambda}(M) \to S_{\lambda'}(M)$ is the map that forgets all of the particles labeled by the number 1. The stabilization map induces a map from the Serre spectral sequence for this fibration (multiplied with $\bR^d$) to the Serre spectral sequence for the version of this fibration associated to $S_{1\, \lambda}(M)$. The map on base spaces is homotopic to the identity.

The result now follows by spectral sequence comparison and the homological stability ranges of Church and Randal-Williams: a range $*<i$ for all dimensions $\geq 2$ from Corollary 3 of \cite{Ch}, a range $*\leq i$ for all dimensions $\geq 3$ from Theorem B of \cite{RW} and the improved range with vanishing reduced Betti numbers from Proposition 4.1 of \cite{Ch}. 

Two remarks are in order. Firstly, Church's results concern the transfer map, not the stabilization map. This is not an issue as Lemma \ref{lemtransferinverse} shows that the stabilization maps are always injective and injective maps between isomorphic vector spaces of finite dimension are isomorphisms. Assuming that our manifolds are the interior of a manifold with a finite handle decomposition guarantees the relevant homology groups are finite-dimensional. This assumption can be removed either by noting that Lemma \ref{lemtransferinverse} proves the stronger statement that the stabilization map and transfer map are inverse up to an automorphism in the stable range (which is also proven in Section 7 of \cite{RW} using Lemma 2.2 of \cite{Do}), or by exhausting the manifold by submanifolds with finite handle decompositions. Secondly, if $M$ has the property that $\tilde H_i(M;\bQ) = 0$ for $ i \leq a$ with $a < \dim M -1$, then by Mayer-Vietoris $M \backslash \{r \text{ points}\}$ has the same property. In other words, removing points from a manifold preserves conditions $(*)_a$.\end{proof}

The spaces $\symw_\lambda(M)$ have a filtration whose filtration differences are disjoint unions of spaces of the form $\syms_{\lambda'}(M)$. We will now describe this filtration. An \textit{elementary collapse} of a partition $\lambda$ is a partition $\lambda'$ which is identical to $\lambda$ except that two integers have been replaced by their sum. A partition $\lambda'$ is a \textit{collapse} of $\lambda$ if $\lambda'$ can be constructed from $\lambda$ by a sequence of elementary collapses. In this case we write $\lambda' \leq \lambda$. For example $1 + 2 + 2 + 4 \leq 1 + 1 + 1 + 2 + 4$. For $\lambda$ a partition of $k$, let $\rm{col}_p(\lambda)$ be the set of collapses of $1^k$ by $p$ elementary collapses that are not collapses of $\lambda$. This is well-defined because even though there are many possible sequences of elementary collapses, the number of elementary collapses from $1^k$ to a given $\lambda$ is always the same. In particular,  $\rm{col}_p(\lambda)$ is the set of partitions $\lambda' \not\leq \lambda$ with cardinality $k - p$. 

\begin{example} If $\lambda = 1 + 3$, then $\rm{col}_0(\lambda) = \{ 1 + 1 + 1 + 1 \}, \rm{col}_1(\lambda) = \{1 + 1 + 2\}$, $\rm{col}_2(\lambda)= \{2+2\}$, and $\rm{col}_p(\lambda) =  \emptyset $ for $p \geq 3$. In particular, $1 + 3 \not\in \rm{col}_2(\lambda)$.
\end{example}

The filtration of $\symw_\lambda(M)$ that we will describe will have filtration differences whose components are indexed by the sets $\rm{col}_p$. We remark that $W_\lambda(M)$ as a set is a disjoint union of $S_{\lambda'}(M)$ for $\lambda' \in \rm{col}_q(\lambda)$ with $0 \leq q \leq k-1$. Our filtration will be obtained by only taking those collapses in $\rm{col}_q(\lambda)$ for $0 \leq q \leq p$.

\begin{definition}
\label{secStrata} For a partition $\lambda$, manifold $M$ and $p \geq 0$, let
\[\cS_\lambda[p] = \bigcup_{\lambda' \in \mr{col}_{p}(\lambda)}S_{\lambda'}(M) \qquad \text{and} \qquad 
U_{p} = \bigcup_{q \leq p} \cS_\lambda[q]. \]
\end{definition}

We make several remarks about this filtration. As a space, $\cS_\lambda[p]$ is not just a disjoint union of sets, but is in fact a disjoint union of spaces. Indeed, for the associated partition to change, particles would need to collide or fall apart, going to a higher or lower part of the filtration. If $\dim M \geq 2$ and $M$ is path-connected, then each of the $S_{\lambda'}(M)$ is path-connected, so the path components of $\cS_{\lambda}[p]$ are indexed by $\mr{col}_p(\lambda)$. Furthermore, $U_p$ is also equal to the subspace of $W_\lambda(M)$ such that forgetting the multiplicities of the particles results in a configuration of at least $k-p$ distinct unordered particles. From this one can see that we have an increasing filtration by open subsets, with $U_p=\symw_\lambda(M)$ for $p \geq k-1$,  $U_{0}=S_{1^k}(M)$ and $U_p$ empty for $p < 0$. Crossing with $\R^{d}$ gives a filtration of $\R^{d} \times \symw_\lambda(M)$. Using \cref{specOpen} applied to the filtrations of $ \symw_\lambda(M)$ and $\R^{d} \times \symw_{\lambda}(M)$ gives the following spectral sequences.

\begin{lemma} \label{specEven} There is a first quadrant spectral sequence converging to $H^{p+q}_c(\symw_{\lambda}(M) )$ with $E^1$-page 
\[E^1_{p,q} = H_c^{p+q}(\cS_{\lambda}[p] ).\]	
Similarly, there is a first quadrant spectral sequence converging to $H^{p+q}_c(\mathbb{R}^{d} \times \symw_{\lambda}(M))$ with $E^1$-page
\[{'E}^1_{p,q} = H_c^{p+q}(\mathbb{R}^{d} \times \cS_{\lambda}[p]). \]
\end{lemma}

The next lemma says that stabilization maps $\stab_*:H_c^*(\bR^d \times \syms_{\lambda}(M)) \m H_c^*(\syms_{1\,\lambda}(M))$ induce a map between these spectral sequences.

\begin{lemma} \label{lemStratamap} There is a map between spectral sequences of Lemma \ref{specEven} that on the $E^1$-page is induced by the stabilization maps $\stab_*:H^*_c(\bR^d \times \syms_{\lambda'}(M)) \m H^*_c(\syms_{1\,\lambda'}(M))$
and on the $E^\infty$-page is the associated graded of $\stab_*:H_c^*(\bR^d \times \symw_{\lambda}(M)) \m H_c^*(\symw_{1\,\lambda}(M))$.
\end{lemma}
\begin{proof} Recall that open embeddings induce maps on compactly supported cohomology via extension by zero. The stabilization map is an open embedding compatible with the filtrations used in \cref{specEven}.
\end{proof}

\begin{lemma} \label{lemStratabound} Let $\lambda$ be a partition of $k$. The stabilization map induces a bijection between the sets of path components of $\cS_{1^j\lambda}[p]$ and $\cS_{1^{j+1}\lambda}[p]$ for $p \leq \frac{j + k}{2}$. 
\end{lemma}
\begin{proof} Recall that the path components of $\cS_{1^{j}\lambda}[p]$ and $\cS_{1^{j+1}\lambda}[p]$ are indexed by $\rm{col}_p(1^j \lambda)$ and $\rm{col}_p(1^{j+1} \lambda)$ respectively. The map between these path components induced by the stabilization map is clearly injective for all $i$. It is not surjective only when $\cS_{1^{j+1} \lambda}[p]$ has a component of the form $\syms_{\lambda''}(M)$ where $\lambda''$ is a partition without any $1$'s. On the other hand if $\syms_{\lambda''}(M) \subset \cS_{1^{j+1}\lambda}[p]$ then $\lambda''$ contains at least $j+k + 1 - 2p$ copies of $1$. Therefore, the map of components is surjective when $j + k - 2p \geq 0$ or equivalently $p \leq \frac{j + k}{2}$.
\end{proof}

It is helpful to combine Lemma \ref{lemOpenstratastab} and Lemma \ref{lemStratabound} into a single statement.

\begin{lemma}\label{lemstabrangeopencombined} Let $\lambda$ be a partition of $k$. The map $t_*: H_*(\R^d \times \cS_{1^j \lambda}[p];\bQ) \to  H_*(\cS_{1^{j+1} \lambda}[p];\bQ)$ is an isomorphism in the following ranges: (i) $* \leq k+j-2p$ if $M$ is of dimension $d>2$, (ii) $* < k+j-2p$ if $M$ is of dimension $d=2$, and (iii) $* < (a+1)(k+j-2p)$ if condition $(*)_a$ holds.\end{lemma}

The following general lemma about spectral sequences follows from the five lemma (see e.g. Remarque 2.10 of \cite{cdg} for a similar argument).

\begin{lemma} \label{lemSSbound} Let $f$ be a map of spectral sequences $\{E^r_{p,q}\} \to \{{'E}^r_{p,q}\}$. If $f:E^1_{p,q} \rightarrow {'E}^1_{p,q}$ is a surjection for $p+q \geq *-1$ and an isomorphism for $p + q \geq *$ then $f: E^\infty_{p,q} \rightarrow {'E}^\infty_{p,q}$ is a surjection for $p+q \geq *-1$ and an isomorphism for $p+q \geq *$.
\end{lemma}

We now prove \cref{propevenopen}.

\begin{proof}[Proof of \cref{propevenopen}]
Thinking of ${\rm Sym}_k(M)$ as a quotient of $M^k$ by $\fS_k$, we note that we can construct the spaces $\symw_{1^j \lambda}(M)$ as quotients of open submanifolds $\tilde{\symw}_{1^j \lambda}(M)$ of the orientable manifold $M^k$ by finite groups acting via orientation preserving maps (this uses that the dimension of $M$ is even). Thus $\symw_{1^j \lambda}(M)$ satisfies rational Poincar\'{e} duality, as in particular it is the underlying space of an orientable orbifold. Theorem V.9.2 of \cite{bredonsheaf} proves Poincar\'e duality for sheaf (co)homology with coefficients in a local system $\cL$ for $M$ a weak homology manifold with respect to $\cL$. Underlying spaces of orbifolds are always weak homology manifolds with respect to locally constant rational local systems, which are trivial if the orbifold is orientable. By Remark \ref{remarksheaf} we can identify compactly supported sheaf and singular cohomology.

Using rational Poincar\'{e} duality and the fact that it commutes with open embeddings, we see that it suffices to show that the stabilization map \[ \stab_*: H^*_c(\mathbb{R}^{d} \times \symw_{1^j\lambda}(M); \mathbb{Q}) \m H^*_c(\symw_{1^{j+1}\lambda}(M); \mathbb{Q})\] is an isomorphism for $* > \dim(\symw_{1^{j+1}\lambda}(M) ) -  f_{M,k}(j)$ and a surjection for $* = \dim(\symw_{1^{j+1}\lambda}(M)) - f_{M, k}(j)$. By \cref{specEven}, there are spectral sequences
	\begin{align*} {'E}^1_{p,q} &= H_c^{p+q}(\mathbb{R}^{d} \times \cS_{1^j\lambda}[p] ; \mathbb{Q}) \Rightarrow H^{p+q}_c(\mathbb{R}^{d}\times \symw_{1^j \lambda}(M) ;  \mathbb{Q})\\
	{E}^1_{p,q} &= H_c^{p+q}({\cS}_{1^{j+1}\lambda}[p] ; \mathbb{Q}) \Rightarrow H^{p+q}_c(\symw_{1^{j+1} \lambda}(M) ;  \mathbb{Q})\end{align*}
By \cref{lemStratamap}, there is a map of spectral sequences $\stab_{p,q}: {'E}^1_{p,q} \rightarrow {E}^1_{p,q}$ given by stabilization maps, which on $E^\infty$ is the associated graded of the stabilization map $\stab_*:H^*_c(\mathbb{R}^{d} \times \symw_{1^j\lambda}(M); \mathbb{Q}) \m H^*_c(\symw_{1^{j+1}\lambda}(M); \mathbb{Q})$.

First consider the case of manifolds $M$ of dimension $d=2n > 2$. Take $p$ and $q$ with $p+q \geq \dim(\symw_{1^{j+1}\lambda}(M) ) -  f_{M,k}(j)=d(j+k+1)-(j+k)$. To apply Lemma \ref{lemSSbound}, we should prove that the stabilization map \[t_{p,q}: H^{p+q}_c(\bR^{d} \times \cS_{1^j \lambda}[p];\bQ) \to H^{p+q}_c(\cS_{1^{j+1} \lambda}[p];\bQ)\] is an isomorphism or surjection in the desired degrees. Using Poincar\'e duality this is equivalent to the map \[H_{d(j+k+1-p)-p-q}(\R^d \times \cS_{ 1^j \lambda}[p];\bQ) \to H_{d(j+k+1-p)-p-q}(\cS_{1^{j+1} \lambda}[p];\bQ)\] being an isomorphism or surjection in the desired degrees, because the $\cS_{\lambda}(M)$ are orientable manifolds if $d$ is even and $M$ orientable. From Lemma \ref{lemstabrangeopencombined} and the discussion following it, we know that this map on homology is an isomorphism if $d(j+k+1-p)-p-q \leq j+k-2p$, i.e. if $(d-1)p+q \geq d(j+k+1)-(j+k)$. This follows directly from $p+q \geq d(j+k+1)-(j+k)$ since $d-1 \geq 1$. We conclude that the map ${'E}^1_{p,q} \rightarrow {E}^1_{p,q}$ is an isomorphism for $p+q \geq d(j+k+1)-(j+k)$. By \cref{lemSSbound}, the stabilization map \[ \stab_*:H^*_c(\mathbb{R}^d \times W_{1^j \lambda}(M); \mathbb{Q}) \m H^*_c(W_{1^{j+1} \lambda}(M); \mathbb{Q})\] is an isomorphism for $* > \dim(\symw_{1^{j+1}\lambda}(M)) - (j + k)$ and a surjection for $* = \dim(\symw_{1^{j+1}\lambda}(M)) - (j + k)$.

By using the homological stability ranges from Lemma \ref{lemstabrangeopencombined} for manifolds of dimension $2$ or manifolds satisfying condition $(*)_a$, we obtain the other stability ranges.
\end{proof}

\begin{remark}\label{remintegral}
If $\dim M >2$, we are forced to work with rational coefficients because the spaces $W_{1^j \lambda}(M)$ do not have integral \p duality. However, if $\dim M=2$ we can prove stability integrally, because $W_{1^j \lambda}(M)$ is a manifold and hence has integral \p duality. To see that $W_{1^j \lambda}(M)$ is a manifold it suffices to remark it is an open subset of ${\rm Sym}_{j+k}(M)$ and the latter is a manifold because locally it is a product of spaces of the form ${\rm Sym}_{m}(\bC) \cong \bC^m$ using the relationship between the coefficients of a monic complex polynomial and its roots.

After remarking this, almost all of the arguments of this section go through unchanged to prove that $H_*(W_{1^j \lambda}(M);\bZ)$ stabilizes if $\dim M=2$. The one modification needed is the following. Instead of using the results of Church in \cite{Ch} and Randal-Williams in \cite{RW} on rational homological stability for $S_{1^j}(M)$, we use integral homological stability, see e.g. Segal \cite{Se} or Randal-Williams \cite{RW}. In particular, in Proposition A.1 of \cite{Se} Segal proved that $t:S_{1^j }(M) \m S_{1^{j+1}}(M)$ induces an isomorphism on integral homology for $* \leq j/2$. Therefore, we have that $t_*:H_*(W_{1^j \lambda}(M);\bZ) \m H_*(W_{1^{j+1}\lambda}(M);\bZ)$ is an isomorphism for $* \leq  (j+k)/2$ if $\dim M = 2$.

When we consider closed manifolds in Section \ref{secpuncturing}, the use of rational coefficients will also be unavoidable. In fact, from the presentation of the spherical braid group given in \cite{FV}, one sees that $H_1(W_{1^j2}(\C P^1);\bZ)=\bZ/(2j+2)\bZ$. Hence integral homological stability fails for closed manifolds, even in dimension two.

We also note that these techniques show that the spaces $S_{1^j \lambda}(M)$ and $W_{1^j \lambda}(M)$ have stability for appropriately shifted integral compactly supported cohomology provided $M$ is a connected orientable manifold that is the interior of a manifold with non-empty boundary and has even dimension at least two.
\end{remark}

\section{The proof for general open manifolds }\label{secNON}

The goal of this section is to generalize \cref{propevenopen} to the case of connected open manifolds that are not necessarily orientable or even dimensional. If $M$ is not orientable or not even dimensional then $W_\lambda(M)$ is not necessarily orientable. In the previous section, we often invoked Poincar\'e duality, so to make these proofs work we need to appropriately modify our statements to include local systems.  For discussions of local systems see e.g. Chapter 5 of \cite{kirkdavis} or Section I.7 of \cite{bredonsheaf}.

Since $W_\lambda(M)$ is the underlying space of an orbifold, by Theorem V.9.2 of \cite{bredonsheaf}, there is a rational orientation local system $\OO$ such that 
\[H_*(W_\lambda(M);\cL) \cong H_c^{\dim(W_\lambda(M))-*}(W_\lambda(M);\OO \otimes \cL)\] if $\cL$ is any locally constant rational local system of dimension one (note that one can also write $\cO$ on the left hand side, by replacing $\cL$ with $\cO \otimes \cL$ and remarking that $\cO^{-1} \cong \cO$). To compute $H^*_c(W_\lambda(M);\OO)$, we use the filtration from Definition \ref{secStrata}. On the spaces in the filtration and on the filtration differences, we use local systems which are pull backs of $\OO$ under the inclusions into $W_\lambda(M)$. \cref{specEven} generalizes as follows. 

\begin{lemma} \label{specNON} Let $\iota: \cS_{\lambda}[p] \m W_{\lambda}(M)$ be the inclusion map.  There exists a first quadrant spectral sequence converging to $H^{p+q}_c(\symw_{\lambda}(M) ; \OO)$   with $E^1$-page  $E^1_{p,q} = H_c^{p+q}(S_{\lambda}[p] ; \iota^* \OO)$. A similar spectral sequence also exists computing $H^{p+q}_c(\R^d \times \symw_{\lambda}(M) ; \OO')$ with $\OO'$ the orientation local system on $\R^d \times W_{\lambda}(M)$.
\end{lemma}

Using these spectral sequences, we see that to prove homological stability for symmetric complements, it suffices to show that $t_*:H_c^{*}(\R^d \times S_{\lambda}(M) ; ({\rm id} \times \iota)^* \OO') \m H_c^{*}(S_{1\,\lambda}(M) ; \iota^* \OO)$ is an isomorphism in a range. The map $t_*$ is obtained from the natural isomorphism of local coefficients $(id \times \iota)^*(\OO') \m \iota^* \OO$ coming from the fact that $t$ is an open embedding. The spaces $S_{\lambda}(M)$ are manifolds and hence have Poincar\'e duality. Let $\OO_{\lambda}$ denote the orientation local system of $S_{\lambda}(M)$ and $\OO_{\lambda}'$ denote the orientation local system on $\R^d \times S_{\lambda}(M)$, which is isomorphic to the pull back of the orientation local system of $S_{\lambda}(M)$. By Poincar\'e duality for the strata, we see that stability for the groups $H_*(S_{1^j\lambda}(M);\OO_{1^j \lambda} \otimes \iota^*\OO )$ is the relevant generalization of \cref{lemOpenstratastab}. 

As before, there is a natural isomorphism of local coefficient systems between $\OO_{\lambda}' \otimes \iota^*\OO'$ and $t^*(\OO_{1\lambda} \otimes \iota^*\OO )$ coming from the open embedding $t$. This isomorphism and the stabilization map give a homomorphism $t_*: H_*(\R^{d} \times S_{\lambda}(M);\OO_{\lambda}' \otimes \iota^*\OO' ) \to H_*( S_{1\,\lambda}(M);\OO_{1\lambda} \otimes \iota^*\OO )$.

\begin{lemma}\label{lemNONstratastab} Let $\lambda$ be a partition with $i$ $1$'s and let $M$ be a connected manifold of dimension $d$ which is the interior of a manifold with non-empty boundary. The map \[t_*: H_*(\R^{d} \times S_{\lambda}(M);\OO_{\lambda}' \otimes \iota^*\OO' ) \to H_*( S_{1\,\lambda}(M);\OO_{1\lambda} \otimes \iota^*\OO )\] is an isomorphism in the following ranges: (i) $* \leq i$ if $M$ is of dimension $d>2$, (ii) $* < i$ if $M$ is of dimension $d=2$ and $M$ is orientable, (iii) $* \leq i$ if $M$ is of dimension $d=2$ and $M$ is not orientable, and (iv) $* < (a+1)i$ if condition $(*)_a$ holds.\end{lemma}

\begin{proof} Since $\OO_\lambda'$ is isomorphic to the pull back of the orientation local system on $S_{\lambda}(M)$, the map in the statement of the lemma is isomorphic to 
\[t_*: H_*(S_{\lambda}(M);\OO_{\lambda} \otimes \iota^*\OO ) \to H_*( S_{1\,\lambda}(M);\OO_{1\lambda} \otimes \iota^*\OO ).\]
Consider the fiber sequence \[S_{1^{i}}(M \backslash \{r \text{ points}\}) \to S_{\lambda}(M) \to S_{\lambda'}(M)\] used in the proof of \cref{lemOpenstratastab}. There is a twisted Serre spectral sequence for this fibration converging to the homology of $H_*(S_{\lambda}(M);\OO_{\lambda} \otimes \iota^*\OO )$, see e.g. Theorem 3.2 of \cite{moerdijksvensson} in the case $G = *$. This has $E_2$-page given by the homology of $S_{\lambda'}(M)$ with coefficients in a graded local system $\mathcal F_*$. This graded local system $\mathcal F_*$ is a bundle of graded $\Q$ vector spaces with fibers isomorphic to $H_*(S_{1^{i}}(M \backslash \{r \text{ points}\});f^*( \OO_{\lambda} \otimes \iota^*\OO) )$. Here $f$ is the inclusion of a fiber into the total space. 

We now explain why the local system $f^*( \OO_{\lambda} \otimes \iota^*\OO)$ is in fact trivial. Pick a base point $\sigma_0$ in $S_{\lambda}(M)$ and consider $\gamma \in \pi_1(S_{\lambda}(M),\sigma_0)$. There are two natural maps $d_1,d_2: \pi_1(S_{\lambda}(M),\sigma_0) \m H_1(M;\bZ)$ given by viewing $\gamma$ as a collection of loops in $M$ and adding their classes in $H_1(M;\bZ)$. The first map $d_1$ remembers the multiplicities of the particles while the second one does not. Let $\OO_M: H_1(M;\bZ) \m \bZ_2=\{1,-1\}$ be the monodromy associated to the orientation local system on $M$. Suppose $\lambda$ has $n_1$ 1's, $n_2$ 2's, etc. Let $p_m: \pi_1(S_{\lambda}(M),\sigma_0) \m  \fS_{n_m}$ be the map that remembers the permutations associated to the paths of particles of multiplicity $m$, which requires an additional choice of ordering on the particles in $\sigma_0$ with multiplicity $m$. Let $s_2:\pi_1(S_{\lambda}(M),\sigma_0)  \m \bZ_2$ be given by $s_2(\gamma)=\prod \epsilon(p_m(\gamma))$ with $\epsilon$ the sign homomorphism. Likewise define $s_1$ by the formula $s_1(\gamma)=\prod \epsilon(p_m(\gamma))^m$. The local system $\iota^* \OO$ can be described as the rational local system of dimension one with monodromy around each loop $\gamma$ given by $\OO_M(d_1(\gamma)) s_1(\gamma)^d$. The local system $\OO_{\lambda}$ on the other hand associates to $\gamma$ the number $\OO_M(d_2(\gamma)) s_2(\gamma)^d$. 

These two local systems agree on loops where only particles with odd multiplicities move. Therefore, they agree on the image of the fundamental group of a fiber and so $f^*( \OO_{\lambda} \otimes \iota^*\OO)$ is the trivial local system $\bQ$. Indeed, in the fiber only particles with multiplicity $1$ move.

The rest of the proof follows the pattern of  \cref{lemOpenstratastab}.  The stabilization map induces a map from the Serre spectral sequence for this fibration cross $\R^d$ to the Serre spectral sequence for fibration associated to $S_{1\, \lambda}(M)$. Since the local systems on the fibers are trivial, homological stability for the spaces $S_{1^i}(M)$ and spectral sequence comparison complete the proof. We note that Church's homological stability result only applies to orientable manifolds and Randal-Williams' range of $* \leq i$ only applies in dimensions $> 2$, so for non-orientable surfaces we need to invoke  Theorem 1.3 of \cite{K}.\end{proof}

The rest of the arguments in \cref{evenSec} apply with little modification to give the following proposition.

\begin{proposition}\label{propallguys} Let $M$ be a connected manifold $M$ of dimension at least $2$ that is the interior of a manifold with non-empty boundary. The stabilization map $t_*: H_i(\R^d \times W_{1^j \lambda}(M);\bQ) \to H_i(W_{1^{j+1} \lambda}(M);\bQ)$ is an isomorphism for $i < f_{M,k}(j)$ and a surjection for $i = f_{M,k}(j)$, where $f_{M,k}$ is the function given in Equation \ref{eqnfmlambda}.
\end{proposition}

\section{The proof for closed manifolds by puncturing}\label{secpuncturing} In this section, we prove homological stability for the spaces $W_{1^j \lambda}(M)$ for $M$ closed. One cannot define stabilization maps for closed manifolds as there is no way to add an extra particle. Instead we use the so-called transfer map. Our proof is similar to that used by Randal-Williams in Section 9 of \cite{RW} to leverage homological stability for configuration spaces of particles in open manifolds to prove homological stability for configuration spaces of particles in closed manifolds. We will first recall the definition of the transfer map $ \tau: H_*(W_{1^{j+1} \lambda}(M);\Q) \to H_*(W_{1^{j} \lambda}(M);\Q)$. We then prove that the transfer map induces an isomorphism in the same range as the stabilization map when the manifold is open. Using an augmented semisimplicial space, we describe a spectral sequence computing $H_*(W_{1^{j} \lambda}(M);\bQ) $ in terms of $H_*(W_{1^{j} \lambda}(N);\bQ) $ with $N$ equal to $M$ minus a finite number of points. The transfer map will induce a map of spectral sequences of this form. The theorem will follow by comparing the spectral sequence for $W_{1^{j} \lambda}(M)$ with the one for $W_{1^{j+1} \lambda}(M)$.

Let $\lambda$ be a partition of $k$. Let $\tilde{W}_{\lambda}(M)$ and $\tilde{S}_{\lambda}(M)$ be ordered versions of the symmetric complement and stratum. That is, these spaces are defined as the inverse image of the projection $M^k \m {\rm Sym}_k(M)$ of the spaces $W_{\lambda}(M)$ and $S_{\lambda}(M)$. For $i \leq j$, let $\mr{del}_{i,j} : \tilde W_{1^j \lambda}(M) \m \tilde W_{1^i \lambda}(M)$ be the map which deletes the last $j-i$ particles of $\tilde W_{1^j \lambda}(M)$. This makes sense since the particles of $\tilde W_{1^j \lambda}(M)$ are ordered and $\lambda' \nleq 1^j\lambda$ implies $\lambda' \backslash (\text{any $j-i$ elements}) \nleq 1^i\lambda$. The map $\mr{del}_{i,j}$ is $\fS_{i+k}$-equivariant. Here $\fS_{i+k}$ acts on $ \tilde W_{1^j \lambda}(M) $ via the inclusion of $\fS_{i+k}$ into $\fS_{j+k}$ induced by the standard inclusion $\fS_i$ into $\fS_j$. Thus it induces a map \[(\mr{del}_{i,j})_*:H_*(\tilde W_{1^{j} \lambda}(M);\bQ)_{\fS_{i+k}} \to H_*(\tilde W_{1^{i} \lambda}(M);\bQ)_{\fS_{i+k}}.\]
Here $V_G$ denotes the coinvariants of a rational representation $V$ of a finite group $G$, defined as $V \otimes_{\bQ[G]} \bQ$ where $G$ acts trivially on $\bQ$. If $X$ is a $G$-space, we have an isomorphism $H_*(X/G;\bQ) \cong H(X;\bQ)_G$. By viewing $\fS_{i+k}$ as a subgroup of $\fS_{j+k}$, for any $\fS_{j+k}$-module $V$ there is a map 
\[V_{\fS_{j+k}} = V \otimes_{\bQ[\fS_{j+k}]} \bQ \to V \otimes_{\bQ[\fS_{j+k}]} \bQ[\fS_{j+k}]/\bQ[\fS_{i+k}] \cong V \otimes_{\bQ[\fS_{i+k}]} \bQ = V_{\fS_{i+k}}\]
induced by the map $\bQ \to \bQ[\fS_{j+k}]/\bQ[\fS_{i+k}]$ sending $1$ to the sum of representatives for the cosets. This construction gives us a map 
\[\iota: H_*(\tilde W_{1^{j} \lambda}(M);\bQ)_{\fS_{j+k}} \to H_*(\tilde W_{1^{j} \lambda}(M);\bQ)_{\fS_{i+k}}\]
 which we can use to define the transfer map. 

\begin{definition}\label{deftransfer}
The transfer map $\tau_{i,j}: H_*(W_{1^{j} \lambda}(M);\bQ) \to H_*(W_{1^{i} \lambda}(M);\bQ)$ is defined as \[ (\mr{del}_{i,j})_* \circ  \iota: H_*(\tilde W_{1^{j} \lambda}(M);\bQ)_{\fS_{j+k}} \m H_*(\tilde W_{1^{i} \lambda}(M);\bQ)_{\fS_{i+k}} \] postcomposed and precomposed with the natural isomorphisms $H_*(W_{1^{i} \lambda}(M);\bQ) \cong H_*(\tilde W_{1^{i} \lambda}(M);\bQ)_{\fS_{i+k}} $ and $H_*(W_{1^{j} \lambda}(M);\bQ) \cong H_*(\tilde W_{1^{j} \lambda}(M);\bQ)_{\fS_{j+k}} $ respectively.

\end{definition}

For $i = j-1$, we denote $\tau_{i,j}$ by $\tau$.  Since we will not use compactly supported cohomology in this section, we drop the factor of $\R^d$ from the domains of stabilization maps. We now show that $\tau$ induces a rational homology equivalence in a range for $M$ open by proving it is an inverse to the stabilization map in the stable range up to an automorphism.

\begin{lemma}\label{lemtransferinverse} Let $M$ be the interior of a connected manifold of dimension at least $2$ with non-empty boundary. The stabilization map induces injections $t_*: H_*(W_{1^{j} \lambda}(M);\bQ) \m  H_*(W_{1^{j+1} \lambda}(M);\bQ) $ in all degrees. The transfer map $\tau: H_*(W_{1^{j+1} \lambda}(M);\bQ) \to H_*(W_{1^{j} \lambda}(M);\bQ)$ is an isomorphism in the range $* \leq  f_{M,k}(j)$. Moreover,  in all homological degrees, we have that $\tau \circ t_*$ is an isomorphism.

\end{lemma}

\begin{proof}Suppose that $\lambda$ is a partition of $k$, then for $-k \leq j \leq 0$ we let $W_{1^j \lambda}(M)$ be the inverse image of $W_{\lambda}(M)$ in $\mr{Sym}_{k+j}(M)$ under $t^{-j}$. Fix $i \geq 0$ and set $B_j = H_i(W_{1^{j-k} \lambda}(M);\bQ)$. We then define $\sigma_j: B_{j-1} \to B_j$ to be the map on homology induced by the stabilization map for $j \geq 1$ and $\sigma_0$ to be $0 \to B_0$. For $p,q \geq k$ the transfer as defined above gives maps $\theta_{q,p}: B_p \to B_q$. The construction of the transfer map can be extended to the remaining cases by restricting the maps $\mr{del}_{i,j}$ to the spaces $B_r$ for $r \leq k$. These satisfy $\theta_{q,p} \circ \sigma_{p-1} = \theta_{q,p-1} + \sigma_{p-1} \circ \theta_{q-1,p-1}$ and $\theta_{p,p} = \mr{id}$. Let $\pi_q$ be the projection $B_q \to B_q/\mr{im}(\sigma_q)$.  By Lemma 2.2 of \cite{Do}, the map  
\[\oplus_{q \leq p} \pi_q \circ \theta_{q,p}: B_p \to \bigoplus_{0 \leq q \leq p} B_q/\mr{im}(\sigma_q)\]
is an isomorphism. In particular, the projection to the $q$-summand of $B_{p-1}$ is given by the map $\pi_q \circ \theta_{q,p-1}$. We now compute the composition of $\theta_{p-1,p} \circ \sigma_p$ with one of these projection maps. To do that, we start by remarking that $\theta_{m,m+1} \circ \ldots \circ \theta_{p-1,p} = (p-m)! \theta_{m,p}$ and thus rationally ${p-q \choose p-m} \theta_{q,p} = \theta_{q,m} \circ \theta_{m,p}$. Next we write
\begin{align*}\pi_q \circ \theta_{q,p-1} \circ \theta_{p-1,p} \circ \sigma_p &= (p-q) \pi_q \circ \theta_{q,p} \circ \sigma_p \\
&= (p-q) \pi_q \circ (\theta_{q,p-1} + \sigma_q \circ \theta_{q-1,p-1}) \\ 
&= (p-q) \pi_q \circ \theta_{q,p-1}\end{align*}

From this we conclude that $\theta_{p-1,p} \circ \sigma_p$ is given by $\oplus_{0 \leq q \leq p} (p-q)$ under the isomorphism $B_p \to \bigoplus_{q \leq p} B_q/\mr{im}(\sigma_q)$. Since we are working rationally, multiplication by a non-zero integer is an isomorphism. Thus we can conclude that $\tau \circ t_*$ is an isomorphism and $t_*$ is an injection. Similarly we conclude that $\theta_{p-1,p}$ is an isomorphism when $\sigma_p$ is. Specializing to $p = k+j+1$ and applying Proposition \ref{propallguys} gives the desired result.
\end{proof}

Using Proposition \ref{propallguys} and Lemma \ref{lemtransferinverse}, we can conclude that $t_*: H_*(W_{1^{j} \lambda}(M);\bQ) \m  H_*(W_{1^{j+1} \lambda}(M);\bQ) $ is an isomorphism on homology for $* =  f_{M,k}(j)$, where $f_{M,k}$ is the function given in Equation \ref{eqnfmlambda}.

Recall that a semisimplicial object is defined as a simplicial object without the data of degeneracy maps. We now describe a semisimplicial space $\tilde{\cW}_\bullet(\lambda)$ with augmentation to $\tilde{W}_{\lambda}(M)$. 

\begin{definition}
The space of $p$-simplices of  $\tilde{\cW}_\bullet(\lambda)$ is given by 
\[\tilde{ \cW_p}(\lambda) = \bigsqcup_{\{m_0,\ldots,m_p\} \in {\tilde S_{1^{p+1}}(M)}} \tilde{W}_{\lambda} (M \backslash \{m_0,\ldots,m_p\})\] The $i$th face map is induced by the inclusion 
\[M \backslash \{m_0,\ldots,m_p\} \m M \backslash \{m_0,\ldots, \hat{m_i}, \ldots, m_p\}\]
where $\hat{m_i}$ indicates that we are omitting the $i$th point. 
\end{definition}

The above construction works equally well for $p=-1$, so that $\tilde \cW_{-1}(\lambda) =\tilde W_{\lambda}(M)$, and making  $\tilde{\cW}_\bullet(\lambda)$ an augmented semisimplicial space. We will show that the augmentation map induces a weak equivalence $||\tilde{\cW}_\bullet(\lambda)|| \to \tilde{W}_{ \lambda}(M)$ with $||\cdot|| $ denoting thick geometric realization. As is customary, we will call $\tilde{\cW}_\bullet(\lambda)$ a resolution of $\tilde W_\lambda(M)$. This resolution is useful as $M \backslash \{m_0,\ldots,m_p\}$ is an open manifold and so we will be able to apply Lemma \ref{lemtransferinverse} levelwise to a semisimplicial chain complex constructed from $\tilde{\cW}_\bullet(\lambda)$. 

To prove that the augmentation is a weak equivalence, we will first recall the definition of a microfibration from  \cite{weissclassify} and the definition of a flag set. We are interested in these definitions since every microfibration with weakly contractible fibers is a weak equivalence and there is an easily checked condition for the contractibility of the geometric realization of a flag set.

\begin{definition}A map $f: E \to B$ is called a microfibration if for $m \geq 0$ and each commutative diagram 
\[\xymatrix{\{0\} \times D^m \ar[r] \ar[d] & E \ar[d] \\
[0,1] \times D^m \ar[r] &B}\]
there exists an $\epsilon \in (0,1]$ and a partial lift $[0,\epsilon] \times D^m \to E$ making the resulting diagram commute.\end{definition}

The following proposition was proven by Weiss in Lemma 2.2 of \cite{weissclassify}.

\begin{proposition}
 A microfibration with weakly contractible fibers is a weak equivalence.
\end{proposition}

We now define flag sets, a type of semisimplicial set where $p$-simplices are defined by their ordered sets of vertices. 


\begin{definition}
A semisimplicial set $X_\bullet$ is said to be a flag set if the natural map $X_p \to X_0^{p+1}$ sending a $p$-simplex to its vertices is an injection and if an ordered $(p+1)$-tuple $(v_0,\ldots,v_p)$ forms a $p$-simplex if and only if $(v_i,v_j)$ forms a $1$-simplex for all $i \neq j$.
\end{definition}

\begin{lemma}\label{lemflagcontr} Let $X_\bullet$ be a flag set such that for each finite collection $\{v_1,\ldots,v_N\}$ of $0$-simplices there exists a $0$-simplex $v$ such that $(v_i,v)$ is a $1$-simplex for all $i$. Then $||X_\bullet||$ is weakly contractible.\end{lemma}

\begin{proof}Let $f: S^i \to ||X_\bullet||$ be an arbitrary continuous map. By simplicial approximation, we can homotope $f$ to a map $g$ which is simplicial with respect to some PL-triangulation of $S^i$. Note that the image of $g$ is contained in the geometric realization of a finite subsemisimplicial set $X'_\bullet$ of $X_\bullet$ spanned by some set of $0$-simplices $\{v_1,\ldots,v_N\}$. By hypothesis the join $||X'_\bullet|| * \{v\}$ is a subcomplex of $||X_\bullet||$ and thus we can extend the map $g$ to a map $\mr{Cone}(S^i) \to ||X_\bullet||$ by sending the cone point to $v$.\end{proof}

We now prove that the augmentation map of our resolution is a weak equivalence. 

\begin{proposition}
The augmentation induces a weak equivalence $||\tilde{\cW}_\bullet(\lambda)|| \to \tilde{W}_{ \lambda}(M)$.
\end{proposition}

\begin{proof}

We will prove that this map is a microfibration with contractible fibers. To see that it is a microfibration, suppose we have a map $f: D^n \times [0,1] \to \tilde{W}_{ \lambda}(M)$ and a lift $\hat{f}$ to $||\tilde{\cW}_\bullet(\lambda)||$ defined on $D^n \times \{0\}$. Given a point $y \in \tilde{W}_{ \lambda}(M)$, the extra data needed to lift $y$ to $||\tilde{\cW}_\bullet(\lambda)||$ is a non-negative integer $p$, a simplicial coordinate $t \in \mr{int}(\Delta^p)$ and a configuration $(m_0,\ldots,m_p) \in \tilde S_{1^{p+1}}(M)$ such that $m_0,\ldots,m_p$ are disjoint from the particles in the configuration $y$. 

Note that if $y'$ is sufficiently close to $y$, the points $m_0,\ldots,m_p$ will also be disjoint from $y'$. Therefore, the data used to lift $y$ will also define a lift of $y'$. Using this idea, we attempt to lift $\hat{f}$ to $(x,s) \in D^n \times [0,1]$ with $s>0$ by setting $\hat{f}(x,s)$ to be the element $f(x,s)$ of $\tilde{W}_{ \lambda}(M)$ together with the simplicial coordinate and configuration associated to $\hat f(x,0)$. For each $x \in D^n$ this is well-defined for $s \leq \epsilon_x$ for some $\epsilon_x > 0$ and by compactness of $D^n$, we can find a single choice of $\epsilon>0$ such that this construction defines a lift on all of $ D^n \times [0,\epsilon]$.

Next we will show that the fibers of the augmentation map are weakly contractible. Note that the fiber of the augmentation map over a configuration $y \in \tilde{W}_{ \lambda}(M)$ is homeomorphic to the geometric realization of the following flag set  $F_{\bullet}(y)$. The set of $p$-simplices of $F_{\bullet}(y)$ is the underlying set of $\tilde S_{1^{p+1}}(M \backslash y)$ and the face maps are induced by forgetting the $i$th point. It is clear that $F_{\bullet}(y)$ is a flag set. It satisfies the conditions of Lemma \ref{lemflagcontr} since we can always find a point in $M$ not contained in some fixed finite subset. Therefore the fibers are weakly contractible and so the augmentation is a weak equivalence.
\end{proof}

We now prove that the transfer map induces a homology equivalence in a range for arbitrary connected manifolds of dimension at least $2$.

\begin{theorem}Let $M$ be any connected manifold of dimension at least $2$. The transfer map $\tau: H_i(W_{1^{j+1} \lambda}(M);\bQ) \to H_i(W_{1^{j} \lambda}(M);\bQ)$ is an isomorphism for $i \leq f_{M,k}(j)$, where $f_{M,k}$ is the function given in Equation \ref{eqnfmlambda}.\end{theorem}

\begin{proof} Note that the transfer map is not a map of spaces so we will be forced to work in the category of chain complexes. Applying rational singular chains to the resolution and using that geometric realization commutes with singular chains up to quasi-isomorphism gives us a semisimplicial chain complex $C_*(\tilde{\cW}_\bullet(1^j \lambda);\bQ)$ such that the augmentation $||C_*(\tilde{\cW}_\bullet(1^j \lambda);\bQ)|| \to C_*(\tilde{W}_{1^j \lambda}(M);\bQ)$ is a quasi-isomorphism.  Here $||\cdot|| $ is defined by taking a semisimplicial chain complex, forming a double complex using the alternating sum of face maps to construct a second differential and then taking the total complex. Applying $\fS_{k+j}$-coinvariants levelwise we get a semisimplicial chain complex with augmentation to a chain complex with homology $H_*(W_{1^j \lambda}(M);\bQ)$ and level $p$ having homology given by
\[\bigoplus_{\{m_0,\ldots,m_p\} \in \tilde S_{1^{p+1}}(M)} H_*(W_{1^j \lambda}(M \backslash \{m_0,\ldots,m_p\});\bQ)\]

\noindent  Note that we have used that, in characteristic zero, taking coinvariants by a finite group action is exact. Applying the construction of the transfer map levelwise to the augmented semisimplicial chain complex gives us a semisimplicial chain map $(\tau_\bullet)_*: C_*(\cW_\bullet(1^{j+1} \lambda);\bQ) \to C_*(\cW_\bullet(1^j \lambda);\bQ)$ inducing the transfer map on homology levelwise. Recall that there is a spectral sequence converging to the homology of the geometric realization of a semisimplicial chain complex with $E^1$-page the homology of the levels. If $A_\bullet$ is a semisimplicial chain complex, this spectral sequence has $E^1_{pq}=H_q(A_p)$. The map $(\tau_\bullet)_*$ induces a map between the spectral sequence for $C_*(\cW_\bullet(1^{j+1} \lambda);\bQ) $ and $C_*(\cW_\bullet(1^j \lambda);\bQ)$. Since the $M \backslash \{m_0,\ldots,m_p\}$ are connected manifolds that are the interior of a manifold with non-empty boundary, Lemma \ref{lemtransferinverse} implies that $(\tau_\bullet)_*$ induces an isomorphism on $E^1_{p,q}$ for $q \leq f_{N_p,k}(j)$ with $N_p$ equal to $M$ with $p$ points removed. Observe that $ f_{M,k}(j) =  f_{N_p,k}(j) $ for all $p$ (this follows from  using Mayer-Vietoris to see that condition $(*)_a$ is preserved by removing points). By a spectral sequence comparison theorem, the transfer map $\tau: H_*(W_{1^{j+1} \lambda}(M);\bQ) \to H_*(W_{1^j \lambda}(M);\bQ)$ is an isomorphism in the range $* \leq  f_{M,k}(j)$.\end{proof}

\begin{remark} \label{hodgeremark}
Since the maps which permute the ordering, forget the ordering and forget a particle are algebraic maps, both $(\mr{del}_{i,j})_*$ and $\iota$ preserve mixed Hodge structures and hence so does the transfer map $\tau: H_i(W_{1^{j+1} \lambda}(M);\bC) \to H_i(W_{1^{j} \lambda}(M);\bC)$. Therefore, if $M$ is algebraic, in the stable range the transfer map induces an isomorphism of mixed Hodge structures.
\end{remark}

\bibliographystyle{amsalpha}
\bibliography{SymComp}

\end{document}